\documentclass[11pt]{amsart}

\usepackage{amsfonts, amstext, amsmath, amsthm, amscd, amssymb}
\usepackage{epsfig, graphics, psfrag}
\usepackage{color}
\renewcommand{\setminus}{{\smallsetminus}}

\newcommand{\ZZ}{{\mathbb{Z}}}

\newcommand{\QQ}{{\mathbb{Q}}}

\newcommand{\bdy}{{\partial}}

\newcommand{\G}{{\mathbb{G}}}
\newcommand{\GA}{{\mathbb{G}_A}}
\newcommand{\GB}{{\mathbb{G}_B}}

\theoremstyle{plain}
\newtheorem{theorem}{Theorem}%[section]

\newtheorem{lemma}[theorem]{Lemma}

\newtheorem*{namedtheorem}{\theoremname}
\newcommand{\theoremname}{testing}

\theoremstyle{definition}
\newtheorem{define}[theorem]{Definition}

\newtheorem{example}[theorem]{Example}

\begin{document}
\title[Slopes and colored Jones polynomials of adequate knots]{Slopes
and colored Jones polynomials \\ of adequate knots}

\author[D. Futer]{David Futer}
\author[E. Kalfagianni]{Efstratia Kalfagianni}
\author[J. Purcell]{Jessica S. Purcell}

\address[]{Department of Mathematics, Temple University,
Philadelphia, PA 19122}

\email[]{dfuter@temple.edu}

\address[]{Department of Mathematics, Michigan State University, East
Lansing, MI, 48824}

\email[]{kalfagia@math.msu.edu}

\address[]{ Department of Mathematics, Brigham Young University,
Provo, UT 84602}

\email[]{jpurcell@math.byu.edu }

\thanks{{E.K. is supported in part by NSF grant DMS--0805942.}}

\thanks{{J.P. is supported in part by NSF grant DMS--0704359.}}

\thanks{ \today}

\begin{abstract}
Garoufalidis conjectured a relation between the boundary slopes of a
knot and its colored Jones polynomials. According to the conjecture, certain boundary slopes
are detected by the sequence of degrees of the colored Jones polynomials. 
We verify this conjecture for \emph{adequate} knots, a class
that vastly generalizes that of alternating knots.
\end{abstract}

\maketitle

\section{Introduction}\label{sec:intro}
For a knot $K \subset S^3$, let $N_K$ denote a tubular neighborhood of
$K$ and let $M_K:=\overline{ S^3\setminus N_K}$ denote the exterior of
$K$. Let $\langle \mu, \lambda \rangle$ be the canonical
meridian--longitude basis of $H_1 (\bdy N_K)$.  An element $p/q \in
{\QQ}\cup \{ 1/0\}$ is called a \emph{boundary slope} of $K$ if there
is a properly embedded essential surface $(S, \bdy S) \subset (M_K,
\bdy N_K)$, such that every circle of $\bdy S$ is homologous to $p \mu + q \lambda \in
H_1 (\bdy N_K)$.  Hatcher has shown that every knot $K \subset S^3$
has finitely many boundary slopes \cite{hatcher:boundary-slopes}.

The \emph{colored Jones function} of $K$ is a sequence of Laurent
polynomial invariants $J_K(n, q) \in {\ZZ}[q, \ q^{-1}]$, for $n=1,2,
\ldots$. For $n=2$, $J_K(2,q)$ is the ordinary Jones polynomial. We
will use the normalization that $J_{\rm unknot}(n, q)=1$, for every $n
\geq 1$.  Let $j(n)$ denote the highest degree of $J_K(n, q)$ in $q$,
and let $j^{*}(n)$ denote the lowest degree. Consider the sequences
$$js_K:= \left\{{{ 4j(n)}\over {n^2} } \: : \: n > 0\right\} \quad
 \mbox{and} \quad js^*_K:= \left\{ {{ 4j^{*}(n)}\over {n^2}} \: : \: n
 > 0 \right\}.$$
Garoufalidis conjectured \cite{garoufalidis:jones-slopes} that for each knot $K$, every cluster point
(i.e., every limit of a subsequence) of $js_K$ or $js^*_K$ is a
boundary slope of $K$.   Thus, if the conjecture holds, the colored Jones polynomials detect certain boundary slopes of $K$. 
He verified
the conjecture for alternating knots, torus knots, pretzel knots of
type $(-2, 3, p)$, and several low crossing knots.

In this paper, we prove Garoufalidis' conjecture for the class of
adequate knots. The precise definition of \emph{adequate} appears in
Section \ref{sec:state-graphs}. For the moment, we note that the
family of adequate knots includes all alternating knots, most
Montesinos knots, and all knots that are Conway sums of two
\emph{strongly alternating} tangles. See \cite{lick-thistle:adequate-knots} and Section
\ref{sec:examples} for more examples.

\begin{theorem}\label{thm:main}
Let $D(K)$ be a knot diagram. Then
\begin{itemize}
\item[(a)] If $D$ is $A$--adequate, then $\displaystyle{ \lim_{n \to
\infty} 4\, n^{-2}j^*(n)}$ exists, and is a boundary slope of $K$.
\item[(b)] If $D$ is $B$--adequate, then $\displaystyle{\lim_{n \to
\infty} 4\, n^{-2}j(n)}$ exists, and is a boundary slope of $K$.
\end{itemize}
In particular, if $K$ is a non-trivial adequate knot, then the set
$js_K \cup js^*_K$ has exactly two cluster points, both of which are integer 
boundary slopes of $K$.
\end{theorem}

The proof of Theorem \ref{thm:main} involves four steps:
\begin{enumerate}
\item Starting with an $A$--adequate diagram $D(K)$, construct a
\emph{state surface} $S_A$ whose boundary is $K$. This is a standard
construction, generalizing the construction of a checkerboard surface.
\item Verify that $S_A$ is an essential surface. This result, stated as
Theorem \ref{thm:essential} below, was first proved by Ozawa
\cite{ozawa:adequate}; an alternate proof is given by the authors in \cite{fkp:guts}.
\item Relate the boundary slope of $S_A$ to the number of positive and
negative crossings in the diagram $D$. This is carried out in Lemma
\ref{lemma:surface-slope}.
\item Relate the limit of $js^*_K$ to the the number of positive and
negative crossings in the diagram $D$. This is carried out in Lemma
\ref{lemma:adequate-limit}.
\end{enumerate}
Taken together, Theorem \ref{thm:essential} and Lemmas
\ref{lemma:surface-slope} and \ref{lemma:adequate-limit} immediately
imply Theorem \ref{thm:main}(a). Part (b) of the theorem follows by
considering the mirror image of the diagram $D$.

\section{State graphs and surfaces}\label{sec:state-graphs}
Let $D$ be a link diagram, and $x$ a crossing of $D$.  Associated to
$D$ and $x$ are two link diagrams, each with one fewer crossing than
$D$, called the \emph{$A$--resolution} and \emph{$B$--resolution} of
the crossing.  See Figure \ref{fig:splicing}.

\begin{figure}[h]
	\centerline{\input{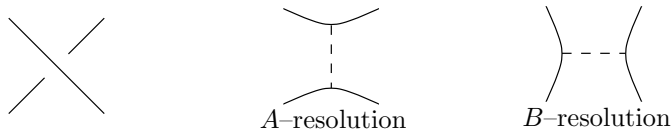}}
\caption{$A$-- and $B$--resolutions of a crossing.}
\label{fig:splicing}
\end{figure}

A Kauffman state $\sigma$ is a choice of $A$--resolution or
$B$--resolution at each crossing of $D$. Corresponding to every state
$\sigma$ is a crossing--free diagram $s_\sigma$: this is a collection
of circles in the projection plane. We can encode the choices that led
to the state $\sigma$ in a graph $\G_\sigma$, as follows. The vertices
of $\G_\sigma$ are in $1-1$ correspondence with the state circles of
$s_\sigma$. Every crossing $x$ of $D$ corresponds to a pair of arcs
that belong to circles of $s_\sigma$; this crossing gives rise to an
edge in $\G_\sigma$ whose endpoints are the state circles containing
those arcs.

Every Kauffman state $\sigma$ also gives rise to a surface $S_\sigma$,
as follows. Each state circle of $\sigma$ bounds a disk in $S^3$. This
collection of disks can be disjointly embedded in the ball below the
projection plane. At each crossing of $D$, we connect the pair of
neighboring disks by a half-twisted band to construct a surface
$S_\sigma \subset S^3$ whose boundary is $K$. See Figure
\ref{fig:statesurface} for an example where $\sigma$ is the all--$A$
state.

Well--known examples of state surfaces include Seifert surfaces (where
the corresponding state $\sigma$ is defined by following an
orientation on $K$) and checkerboard surfaces for alternating links
(where the corresponding state $\sigma$ is either the all--$A$ or
all--$B$ state). In this paper, we focus on the all--$A$ and all--$B$
states of a diagram, but we do not require our diagrams to be
alternating. Thus our surfaces are generalizations of checkerboard
surfaces.

\begin{define}
A link diagram $D$ is called \emph{$A$--adequate} if the state graph
$\GA$ corresponding to the all--$A$ state contains no 1--edge loops.
Similarly, $D$ is called $B$--adequate if the all--$B$ graph $\GB$
contains no 1--edge loops.  A link diagram is \emph{adequate} if it is
both $A$-- and $B$--adequate.  A link that admits an adequate diagram
is also called \emph{adequate}.
\end{define}

\begin{figure}
\includegraphics{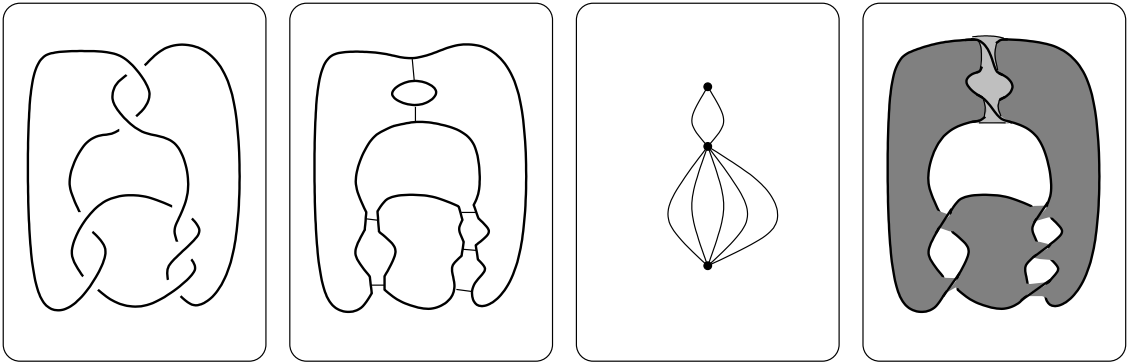}
\caption{Left to right:  A diagram.  The diagram with $A$--resolutions at each crossing.  The
graph $\GA$.  The state
	surface $S_A$.}
\label{fig:statesurface}
\end{figure}

Adequate diagrams are quite common. All reduced alternating diagrams
are adequate. Every $n$--string planar cable of an adequate diagram is
adequate. The standard diagram of a Montesinos link is either
$A$--adequate or $B$--adequate, and typically both. Finally, observe
that a diagram $D(K)$ is $A$--adequate if and only if its mirror image
is $B$--adequate. This observation is useful for the proofs below:
once a result is proved for the all--$A$ state surface $S_A$ of an
$A$--adequate diagram, the corresponding statement about $B$--adequate
diagrams follows by reflection.

\begin{theorem}[Ozawa \cite{ozawa:adequate}]\label{thm:essential} 
Let $D$ be an $A$--adequate diagram of a knot $K$. Then the state
surface $S_A$ is incompressible and $\bdy$--incompressible in the
complement $M_K$. Similarly, if $D$ is a $B$--adequate diagram of a
knot $K$, then $S_B$ is incompressible and
$\bdy$--incompressible.
\end{theorem}

Ozawa's original proof of this theorem relies on building up the surface $S_\sigma$ via Murasugi sums. An
alternate proof from the point of view of normal surface theory will be
given by the authors in \cite{fkp:guts}, where we will also relate the coefficients of 
the colored Jones polynomials $J_K(n,q)$ to the size of the \emph{guts} of the surfaces $S_A$ and $S_B$. These guts can be viewed as the hyperbolic pieces in the geometric decomposition of $S^3 \setminus S_\sigma$. Thus, taken together, \cite{fkp:guts} and Theorem \ref{thm:main} of this paper establish two separate connections between the colored Jones polynomials and classical geometric topology.

\medskip

Recall from the Introduction that if $S \subset M_K$ is a surface such
that $\bdy S$ represents the homology class $p \mu + q \lambda \in
H_1(\bdy M_K)$, we say the boundary slope of $S$ is $p/q \in \QQ \cup
\{ \infty \}$. It turns out that the boundary slope of a state surface
$S_\sigma$ is easy to read from a diagram $D$.

Suppose that $D(K)$ is a diagram of an oriented knot $K$. Then every
crossing of $D$ can be classified as either positive or negative, as
in Figure \ref{fig:local-slope}. For a state $\sigma$ of $D$, let
$c_+^B(\sigma)$ be the number of positive crossings at which $\sigma$
chooses the $B$--resolution. Similarly, let $c_-^A(\sigma)$ be the
number of negative crossings at which $\sigma$ chooses the
$A$--resolution.

\begin{lemma}\label{lemma:state-slope}
Let $D(K)$ be a diagram of an oriented knot $K$, and let $\sigma$ be a
state of $D$. Then the state surface $S_\sigma$ has as its boundary
the slope $2c_+^B(\sigma) - 2c_-^A(\sigma)$.
\end{lemma}

This lemma was observed by Curtis and Taylor for checkerboard surfaces
of alternating knots \cite[Proposition 2.6]{curtis-taylor}. However, both
the statement and the proof hold in complete generality: $S_\sigma$ is
not even required to be an essential surface.

\begin{figure}[h]
\includegraphics{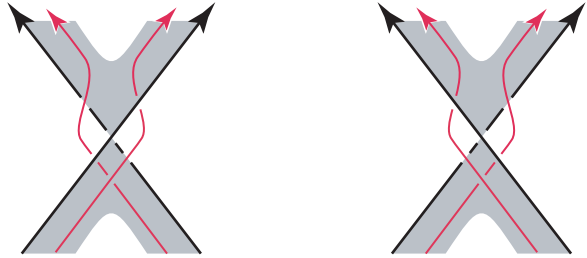}
\caption{Left: a positive crossing, and a piece of state surface
$S_\sigma$ that chooses the $B$--resolution at this crossing. Locally,
this crossing contributes $+2$ to the slope of $S_\sigma$. Right: if
$\sigma$ chooses the $A$--resolution at a negative crossing, the slope
of $S_\sigma$ receives a local contribution of $-2$.}
\label{fig:local-slope}
\end{figure}

\begin{proof}
Suppose, first, that $\sigma$ is the Seifert state, and $S_\sigma$ is
an oriented Seifert surface constructed from the diagram $D$. To
follow an orientation of $K$, $\sigma$ must choose the $A$--resolution
at every positive crossing and the $B$--resolution at every negative
crossing (the opposite of the choices depicted in Figure
\ref{fig:local-slope}). Thus $c_+^B(\sigma) = c_-^A(\sigma) = 0$, by
definition. Also, because $\bdy S_\sigma$ is the boundary of an
orientable surface, this curve is null-homologous in $M_K$ and has
slope $0$. This verifies the lemma for the Seifert state.

Next, let $\sigma$ be an arbitrary state. Then it is still the case
that $S_\sigma$ intersects a meridian of $K$ only once. Thus the
boundary slope of $S_\sigma$ is still an integer $p$. Let $L$ be the
simple closed curve of intersection between $S_\sigma$ and $\bdy M_K$,
oriented in the same direction as $K$. Then the boundary slope $p$ of
$S_\sigma$ is the linking number $lk(K,L)$, or equivalently the
oriented intersection number between $L$ and a Seifert surface for
$K$.

The linking number $lk(K, L)$ can be computed by summing the local
contributions near each crossing. If $\sigma$ chooses the
$A$--resolution at a positive crossing or the $B$--resolution at a
negative crossing, $L$ is locally disjoint from the Seifert surface,
and the local contribution to the linking number is $0$. On the other
hand, for every positive crossing where $\sigma$ chooses the
$B$--resolution, the left panel of Figure \ref{fig:local-slope} shows
that the neighborhood of the crossing contributes $+2$ to the linking
number $lk(K, L)$. Similarly, the right panel of Figure
\ref{fig:local-slope} shows that a negative crossing where $\sigma$
chooses the $A$--resolution contributes $-2$ to $lk(K,L)$.

Adding up these contributions, we conclude that $lk(K, L) =
2c_+^B(\sigma) - 2c_-^A(\sigma)$.
\end{proof}

As a special case of Lemma \ref{lemma:state-slope}, we obtain the
boundary slopes of $S_A$ and $S_B$.

\begin{lemma}\label{lemma:surface-slope}  
Let $D(K)$ be a diagram of an oriented knot $K$. Let $c_+$ be the
number of positive crossings in $D$, and $c_-$ the number of negative
crossings. Then the all--$A$ surface $S_A$ has boundary slope $-2c_-$,
and $S_B$ has boundary slope $2c_+$.
\end{lemma}

\begin{proof}
The all--$A$ state $\sigma$ chooses the $A$--resolution at every
crossing. Thus for the all--$A$ state, $c_+^B(\sigma) = 0$ and
$c_-^A(\sigma) = c_-$, hence $\bdy S_A$ has slope $-2c_-$ by Lemma
\ref{lemma:state-slope}.
Similarly, for the all--$B$ state $\sigma$, $c_+^B(\sigma) = c_+$ and
$c_-^A(\sigma) = 0$, hence $\bdy S_B$ has slope $2c_+$.
\end{proof}

\section{Colored Jones polynomials}\label{sec:colored}
In this section, we relate the degrees of colored Jones polynomials to
the numbers $c_+$ and $c_-$ of positive and negative crossings in a
diagram $D$. A good reference for the following discussion is
Lickorish's book \cite{lickorish:book}.

The colored Jones polynomials of a link $K$ have a convenient
expression in terms of \emph{Chebyshev polynomials}. For $n \geq 0$,
the polynomial $S_n(x)$ is defined recursively as follows:
\begin{equation}\label{eq:cheb-recursive}
S_{n+1} = x S_n - S_{n-1}, \qquad S_1(x) = x, \qquad S_0(x) = 1.
\end{equation}

Let $D$ be a diagram of a link $K$. For an integer $m > 0$, let $D^m$
denote the diagram obtained from $D$ by taking $m$ parallel copies of
$K$.  This is the $m$--cable of $D$ using the blackboard framing; if
$m=1$ then $D^1=D$.  Let $\langle D^m \rangle$ denote the Kauffman
bracket of $D^m$: this is a Laurent polynomial over the integers in a
variable $A$.  Let $w=w(D) = c_+ - c_-$ denote the writhe of
$D$. Then we may define the function
\begin{equation}\label{eq:unreduced}
G(n+1, A):= \left((-1)^n A^{n^2+2n} \right)^{-w} (-1)^{n-1}
\left(\frac{A^4 - A^{-4}}{A^{2n} - A^{-2n}} \right) \langle S_n( D)
\rangle,
\end{equation}
where $S_n(D)$ is a linear combination of blackboard cablings of $D$,
obtained via equation (\ref{eq:cheb-recursive}), and the notation
$\langle S_n(D) \rangle$ means extend the Kauffman bracket linearly.
That is, for diagrams $D_1$ and $D_2$ and scalars $a_1$ and $a_2$,
$\langle a_1 D_1 + a_2 D_2 \rangle = a_1 \langle D_1\rangle +
a_2\langle D_2\rangle$.
For the results below, the important corollary of the recursive
formula for $S_n(x)$ is that
\begin{equation}\label{eq:chebyshev}
S_n(D) = D^n + (1-n)D^{n-2} + \mbox{lower degree cablings of } D.
\end{equation}

Finally, the reduced $(n+1)$--colored Jones polynomial of $K$, denoted
by $J_K(n+1, q)$, is obtained from $G(n+1, A)$ by substituting
$q:=A^{-4}$.

Recall from the Introduction that $j(n)$ denotes the highest degree of
$J_K(n, q)$ in $q$, and $j^{*}(n)$ denotes the lowest degree. Following big--$O$ notation, we let $O(n)$ denote a term that is at most linear in $n$.

\begin{lemma} \label{lemma:adequate-limit}
Let $D$ be a link diagram with $c_+$ positive crossings and $c_-$
negative crossings.
\begin{itemize}
\item[(a)] If $D$ is $A$--adequate, then
$\displaystyle{j^{*}(n)= -  \frac{c_-}{2}n^2+ O(n).}$
\item[(b)] If $D$ is $B$--adequate, then
$\displaystyle{j(n)= \frac{c_+}{2}n^2+ O(n).}$
\end{itemize}
\end{lemma}

\begin{proof} 
For part (a), let $D$ be an $A$--adequate diagram with $c = c(D) = c_+
+ c_-$ crossings. Let $v_A = v_A(D)$ be the number of vertices of
$\GA$, which is equal to the number of state circles in the all--$A$
state. Then, for every $m > 0$, the link diagram $D^m$ is also
$A$--adequate with $c(D^m)=m^2 c$ and $v_A(D^m)=m v_A$.

Let $\deg ( P(A) )$ denote the highest degree of a polynomial $P$ in
$A$.  Then the highest degree of $\langle D \rangle$ is $\deg \langle D \rangle = c+2v_A-2$; see \cite[Lemma
5.4]{lickorish:book} or \cite[Lemma 7.1]{dessin-jones} for a proof.  From equation (\ref{eq:unreduced}),
one can see that $\deg( G(n, \ A) )$ comes from the highest--degree
term of $\langle S_{n-1}(D) \rangle$.  Furthermore, by equation
(\ref{eq:chebyshev}) and the previous paragraph,
$$\deg \, \langle S_{n-1}(D) \rangle \: = \: \deg \, \langle D^{n-1}
\rangle \: = \: (n-1)^2 c + 2(n-1) v_A - 2.$$ Thus
\begin{eqnarray*}
\deg G(n, A) &=& -w(n^2+2n) + (4-2n) + \deg \, \langle S_{n-1}(D) \rangle \\
&=& -w(n^2+2n) + (4-2n) + (n-1)^2 c + 2(n-1) v_A - 2 \\
&=& (c-w) \, n^2 + O(n) \\
&=& ((c_+ + c_-) - (c_+ - c_-)) n^2 + O(n) \\
&=& 2 c_- n^2 + O(n).
\end{eqnarray*}

Finally, since $J_K(n, q)$, is obtained from $G(n, A)$ by substituting
$q:=A^{-4}$, we conclude that the lowest degree of $J_K(n, q)$ in $q$
is $j^{*}(n)= -c_- n^2/2+ O(n)$. This proves (a).

For part (b), it suffices to observe that the mirror image $D^{*}$ of
a $B$--adequate diagram $D$ will be $A$--adequate. Taking the mirror
image also interchanges positive and negative crossings, and replaces
$q$ with $q^{-1}$ in the colored Jones polynomials. Thus the result
follows from (a).
\end{proof}

We can now complete the proof of Theorem \ref{thm:main}.

\begin{proof}[Proof of Theorem \ref{thm:main}]
Let $D(K)$ be an $A$--adequate diagram. Then, by Theorem
\ref{thm:essential}, $S_A$ is an essential surface for $M_K$. By
Lemmas \ref{lemma:surface-slope} and \ref{lemma:adequate-limit}, the
boundary slope of $S_A$ is
$$-2c_- = \lim_{n \to \infty} 4\, n^{-2}j^*(n).$$
Similarly, if $D(K)$ is $B$--adequate, $S_B$ is an essential surface
with boundary slope
$$2c_+ =  \lim_{n \to \infty} 4\, n^{-2}j(n).$$
In particular, if $D(K)$ is adequate, then both sequences $\{ 4\,
n^{-2}j^*(n) \}$ and $\{ 4\, n^{-2}j(n) \}$ converge to boundary
slopes of $K$. For a non-trivial knot $K$, the slopes $-2c_-$ and
$2c_+$ are distinct, because at least one of the integers $c_-$ and
$c_+$ is strictly positive.
\end{proof}

\section{Examples}\label{sec:examples}

\begin{example}
Let $p$ be an odd integer, and let  $K_p$ denote the $(-2, 3, p)$  pretzel knot, with a standard pretzel diagram $D_p$. It is easy to check that $D_p$ is $A$--adequate iff $p>0$ and $B$--adequate iff $p<0$. Furthermore, for all values of $p$ except $p=-1$ (when $D_p$ is an unusual diagram of the $5_2$ knot), the knot $K_p$ does not admit an adequate diagram. This classical fact can also be seen via Theorem \ref{thm:main}, because all Jones slopes of adequate knots are integers.

Garoufalidis computed \cite{garoufalidis:jones-slopes} that
$$ \lim_{n \to \infty} \frac{4j(n)}{n^2} = \frac{2(p^2-p-5)}{p-3} \quad \mbox{if } p \geq 5, \qquad
\lim_{n \to \infty}  \frac{4j^*(n)}{n^2} = \frac{2(p+1)^2}{p} \quad \mbox{if } p \leq -3.$$
From the work of Hatcher and Oertel 
\cite{hatcher-oertel:montesinos-slopes} and
Dunfield \cite{dunfield:boundary-slopes}, it follows that these limiting numbers are indeed boundary slopes of $K_p$. Hence, since all Jones slopes of adequate knots are integers by Theorem \ref{thm:main}, we recover the classical fact that these knots are not adequate.
%
%	Let $K_p$ denote the $(-2, 3, p)$ pretzel knot ($p$ odd).  If $p > 0$,
%	then the usual diagram $D_p$ for $K_p$ is $A$--adequate with $c_-=0$.
%	%% $w(D_p)=c(D_p)$. 
%	Lemma \ref{lemma:adequate-limit} implies that $j^{*}(n) = O(n)$. Thus
%	$4 n^{-2} j^{*}(n) \to 0 $, and $0$ is indeed a slope of $K_p$.
%	Note that, as the formulae in \cite{garoufalidis:jones-slopes} reveal, $K_p$ is not
%	$B$--adequate.
%
%	If $p < 0$, then $D_p$ is $B$--adequate with $c_+(D_p)=5$,
%	$c_-(D_p)=-p$.  Lemma \ref{lemma:adequate-limit} implies that
%	$j(n)={5 \over 2}n^2+ O(n)$, which agrees with the calculations of
%	\cite{garoufalidis:jones-slopes}.  Thus $4 n^{-2} j(n) \to 10 $.  The set of
%	boundary slopes of $K_p$, as computed by Hatcher and Oertel 
%	\cite{hatcher-oertel:montesinos-slopes} and
%	Dunfield \cite{dunfield:boundary-slopes}, is
%	$$\left\{ 0, \ 10, \ {{2(p+1)^2}\over p}, \ 2p+6 \right\}.$$
%	Garoufalidis computes that $ \lim_{n \to \infty} 4\, n^{-2}j^*(n)$  will be $10$ for $p=-1$ and
%	$ {2(p+1)^2}\over p $ for $p\leq -3$.
\end{example}

\begin{example}
Following \cite{lick-thistle:adequate-knots}, a $(2,2)$--tangle $T$ is called \emph{strongly
alternating} if each of the closures of $T$ is a reduced alternating
link diagram.  Any knot obtained as a Conway sum of two
strongly alternating tangles is then adequate (see also \cite{fkp:conway}).
For example, any non-alternating pretzel knot $K(a_1, \ldots, a_r, b_1, \ldots,
b_k)$ with with $a_i, b_j, r, k\geq 2$ is adequate.  Similarly, as
explained in \cite{lick-thistle:adequate-knots}, a reduced diagram of any Montesinos knot with 
at least two positive rational tangles and at least two negative rational tangles will be
adequate.
Theorem \ref{thm:main} implies that these knots satisfy Garoufalidis'
conjecture.
\end{example}

\begin{example}
Let $B_n$ denote the braid group on $n$ strings, and let $\sigma_1, \cdots,
\sigma_{n-1}$ be the elementary braid generators.  Let $D_b$ denote the closed braid diagram
 obtained from the braid
$b=\sigma_{i_1}^{r_1}\sigma_{i_2}^{r_2} \cdots \sigma_{i_k}^{r_k}$. If $r_j > 0$ for all $j$, the positive braid diagram $D_b$ will be $A$--adequate. Since all crossings in this braid are positive, $c_-=0$. Thus, by Lemma \ref{lemma:adequate-limit},
$$ \lim_{n \to \infty} 4\, n^{-2}j^*(n) \:=\: -2c_- \:=\: 0.$$
Furthermore, the essential surface $S_A$ whose boundary is this slope will be a fiber in $S^3 \setminus K$.

Under the stronger hypothesis that $r_j \geq 3$ for all $j$, the diagram $D_b$ is not only $A$--adequate but also $B$--adequate. Thus Theorem 
\ref{thm:main} applies. The other boundary slope detected by the colored Jones polynomials is
$$ \lim_{n \to \infty} 4\, n^{-2}j(n) \:=\: 2c_+ \:=\: 2 \sum_{j=1}^k r_j.$$
\end{example}

\bibliographystyle{hamsplain}
\bibliography{biblio.bib}

\end{document}